\documentclass[12pt]{amsart}

\usepackage{amsmath}
\usepackage{amssymb}

\allowdisplaybreaks                     
\theoremstyle{plain}
\newtheorem{teo}{Theorem}[section]
\newtheorem*{theorem*}{Theorem}
\newtheorem{lem}[teo]{Lemma}
\newtheorem{coro}[teo]{Corollary}
\newtheorem{prop}[teo]{Proposition}
\newtheorem{defi}[teo]{Definition}

\theoremstyle{remark}

\numberwithin{equation}{section}
\newcommand{\quadro}{\hfill\vrule height .9ex width .8ex depth -.1ex}

\newcommand{\NN}{\Bbb{N}}
\newcommand{\ZZ}{\Bbb{Z}}
\newcommand{\RR}{\Bbb{R}}
\newcommand{\CC}{\Bbb{C}}

\newcommand{\nep}{{\rm{e}}}

\newcommand{\CZ}{Calder\'on--Zygmund }

\newcommand{\D}{\Delta} 

\newcommand{\jl}{j_{\ell}}
\newcommand{\di}{\,{\rm{d}}}
\newcommand{\dir}{\,{\rm{d}}\rho}  
\newcommand{\dil}{\,{\rm{d}}\lambda}  

\begin{document}
\title[Spaces $H^1$ and $BMO$ on $ax+b$--groups]{Spaces $H^1$ and $BMO$ on $ax+b$--groups}

\subjclass[2000]{ 22E30, 42B20, 42B30, 46B70}

\keywords{Hardy space, BMO, singular integrals, interpolation}

\thanks{{\bf Acknowledgement.} Work partially supported by the
Italian Progetto cofinanziato ``Analisi Armonica'' 2006--2008.}

\author[M. Vallarino]
{Maria Vallarino}

\address{Maria Vallarino:
Dipartimento di Matematica e Applicazioni
\\ Universit\`a di Milano-Bicocca\\
via R.~Cozzi 53\\ 20125 Milano\\ Italy}
\address{e-mail address: maria.vallarino@unimib.it}

\begin{abstract}
Let $S$ be the group $\RR^d\ltimes \RR^+$ endowed with the Riemannian symmetric space metric $d$ and the right Haar measure $\rho$. The space $(S,d,\rho)$ is a Lie group of exponential growth. In this paper we define an Hardy space $H^1$ and a $BMO$ space in this context. We prove that the functions in $BMO$ satisfy the John--Nirenberg inequality and that $BMO$ may be identified with the dual space of $H^1$. We then prove that singular integral operators whose kernels satisfy a suitable integral H\"ormander condition are bounded from $H^1$ to $L^1$ and from $L^{\infty}$ to $BMO$. We also study the real interpolation between $H^1$, $BMO$ and the $L^p$ spaces.
\end{abstract}

\maketitle
\section{Introduction}
Let $S$ be the group $\RR^d\ltimes \RR^+$ endowed with the product
$$(x,a)\cdot (x',a')=(x+a\,x',a\,a')\qquad\forall (x,a),\,(x',a')\in S\,.$$

We call $S$ an {\emph{$ax+b$-group}}. We endow $S$ with the 
left-invariant Riemannian metric $ds^2=a^{-2}(\di x^2+\di a^2)$. 
We denote by $d$ the corresponding metric, which is that of the 
$(d+1)$-dimensional hyperbolic space.

The group $S$ is nonunimodular; the right and left Haar measures 
are given respectively by
$$\dir(x,a)=a^{-1}\di x\di a\qquad {\rm{and}}\qquad \dil(x,a)=a^{-(d+1)}\di x\di a\,.$$
It is well known that the measure of the ball $B_r$ centred at the 
identity and of radius $r$, behaves like
$$\rho(B_r)=\lambda(B_r)\sim \begin{cases}
r^{d+1}&{\rm{if~}}r<1\\
\nep^{dr}&{\rm{if~}}r\geq 1\,.
\end{cases}
$$
This shows that the space $(S,d,\rho)$ is of {\emph{exponential growth}}. 
Throughout this paper, unless explicitly stated, we consider the right 
measure $\rho$ on $S$ and we denote by $L^p$ the space $L^p(\rho)$ 
and by $\|\cdot\|_p$ the norm in this space, for all $p$ in $[1,\infty] $. 

Harmonic analysis on the space $(S,d,\rho)$ has been the object 
of many investigations, mainly because it is an example of 
exponential growth group, where the classical theory of
singular integral operators does not hold 
(see \cite{CGHM, GQS, GS1, GS, HS, MT}). In this context maximal operators, 
singular integrals and multiplier operators associated with a 
distinguished Laplacian have been studied. In particular, in the case 
when $d=1$, $S$ is the \emph{affine group of the real line}, 
where the theory of singular integrals have been considered by many authors. 

Recently W.~Hebisch and T.~Steger 
\cite{HS} adapted the classical {\emph{\CZ theory}} to the space $(S,d,\rho)$ 
and applied this theory to study singular integral operators 
in this context. The purpose of this paper is to develop a 
$H^1$--$BMO$ theory in the space $(S,d,\rho)$, which is a natural 
development of the \CZ theory introduced in \cite{HS} and 
which may be considered as an analogue of the classical theory. 

The classical $H^1$--$BMO$ theory holds in $(\RR^n,d,m)$, where $d$ is the 
euclidean metric and $m$ denotes the Lebesgue measure. 
In this context the spaces $H^1$ and $BMO$ are defined as 
in \cite{FeS, J, S} and satisfy the following properties:
\begin{itemize} 
\item[(i)] the space $BMO$ may be identified with the dual space of $H^1$;
\item[(ii)] the functions in $BMO$ satisfy the so-called 
John--Nirenberg inequality;
\item[(iii)] the \CZ operators are bounded 
from $H^1$ to $L^1$ and from $L^{\infty}$ to $BMO$;
\item[(iv)] the real interpolation spaces 
between $H^1$ and $BMO$ are the $L^p$ spaces (see \cite{FeS, H, Jo, P, RS}).
\end{itemize}
We recall that there are several characterizations of the Hardy space $H^1$ 
in the classical setting. In particular, an atomic 
definition and a maximal characterization of $H^1$ are available. 
The properties (i)-(iv) involving $H^1$ were proved by using both 
its maximal characterization and its atomic definition.

Extensions of the $H^1$--$BMO$ theory have been considered 
in the literature. In particular, a theory that parallels 
the euclidean theory has been developed in 
{\emph{spaces of homegeneous type}}. A space of homogeneous type 
is a measured metric space $(X,d,\mu)$ where the 
doubling condition is satisfied, i.e., there exists
a constant $C$ such that
\begin{equation}\label{doubling}
\mu\bigl(B(x,2r)\bigr)
\leq C\, \mu\bigl(B(x,r)\bigr)
\qquad\forall x \in X \qquad\forall r \in \RR^+.
\end{equation}
In the space $(X,d,\mu)$ a \CZ theory \cite{CW1,S} and 
a $H^1$--$BMO$ theory \cite{CW2,FS} have been studied. 
This theory is a generalization of the euclidean one; 
in particular properties (i)-(iv) are satisfied.

It is natural to ask whether it is possible to develop a 
$H^1$--$BMO$ theory in spaces which do not satisfy the doubling condition 
(\ref{doubling}). 
This was done in the space $(R^n,d,\mu)$, where $d$ is the euclidean metric 
and $\mu$ is a (possibly nondoubling) measure, which grows polinomially 
at infinity \cite{MMNO, NTV, T}. A space $BMO$ was also 
introduced by A.~Ionescu in symmetric spaces of the noncompact type 
and rank one: note that the $BMO$ theory developed 
in \cite{I} applies to the space $(S,d)$ endowed with the 
Riemannian measure, i.e., the left Haar measure $\lambda$,   
but does not apply to the space $(S,d,\rho)$, 
which we are considering in this paper.

G.~Mauceri and S.~Meda \cite{MM} introduced a $H^1$--$BMO$ theory 
in the space $(\RR^n,d,\gamma)$, where $d$ is the euclidean metric and 
$\gamma$ is the Gauss measure, and applied this theory to study 
appropriate operators related to the Ornstein-Uhlenbeck operator.

In this paper we develop a $H^1$--$BMO$ theory in the space $(S,d,\rho)$ 
defined above. The starting point is the \CZ theory introduced 
in \cite{HS}. There exists a family of appropriate sets in $S$, 
which are called {\emph{\CZ sets}}, which replaces the family 
of balls in the classical \CZ theory. 

For each $p$ in $(1,\infty]$, we define an \emph{atomic Hardy space} 
$H^{1,p}$.  Atoms are functions supported in \CZ sets, with 
vanishing integral and satisfying a certain size condition. 
An important feature of the classical theory is that
all the spaces $H^{1,p}$, for $p$ in $(1, \infty]$, are equivalent. We shall prove that
this holds also in our setting. We define a space of \emph{functions of bounded mean oscillation} $BMO$, whose definition is analogue 
to the classical one, where balls are replaced by \CZ sets. 
We shall prove that the John--Nirenberg inequality is satisfied and that 
$BMO$ may be identified with the dual space of $H^1$.

Further, we show that a singular integral operator, whose kernel satisfies an 
integral H\"ormander condition, extends to a bounded
operator from $H^1$ to $L^1$ and from $L^{\infty}$ to $BMO$. 
As a consequence of this result, we show that spectral multipliers
of a distinguished Laplacian $\D$ extend to bounded operators from $H^1$ to
$L^1$ and from $L^\infty$ to $BMO$.

Finally, we find the real interpolation spaces between $H^1$ and $L^p$, 
$L^p$ and $BMO$, $H^1$ and $BMO$, for $p$ in $(1,\infty)$. 
The interpolation results which we prove are the analogues 
of the classical ones \cite{H,Jo,P,RS}, but the proofs are different. 
Indeed, in the classical setting the maximal characterization 
of the Hardy space is used to obtain the interpolation results, 
while the Hardy space $H^1$ introduced in this paper has only an atomic definition.


Positive constants are denoted by $C$; these may differ from one line to another, and may depend on any quantifiers written, implicitly or explicitly, before the relevant formula. Given two quantities $f$ and $g$, by $f\sim g$ we mean that there exists a constant $C$ such that $1/C\leq f/g\leq C$.

\smallskip
The author would like to thank Nicolas Varopoulos for his help and encouragement.
 
\section{The Hardy space}\label{H1}
In this section, we give the definition of  the Hardy space on $S$, 
where the \CZ sets are involved. Let us recall the 
definition of \CZ sets which appears in \cite{HS} and implicitly in \cite{GS}.
\begin{defi}\label{Czsets}
A \emph{\CZ set} is a set $R=Q\times [a\nep^{-r},a\nep^r]$, where 
$Q$ is a dyadic cube in $\RR^d$ of side $L$, $a\in\RR^+$, $r>0$ and
$$\nep^{2}a\,r\leq L< \nep^{8 }a\,r\qquad{\rm{if~}}r<1\,,$$
$$a\,\nep^{2r}\leq L< a\,\nep^{8r}\qquad{\rm{if~}}r\geq 1\,.$$
\end{defi}
Let $\mathcal R$ denote the family of all \CZ sets. 

In \cite{HS} the authors proved that the space $(S,d,\rho)$ is a 
\CZ space with \CZ constant $\kappa_0$. More precisely, they proved 
that the following hold:
\begin{itemize}
\item[(i)] for every set $R$ in $\mathcal R$ there exist a 
point $x_R$ and a positive number $r_R$ such that $R\subseteq B(x_R,\kappa_0\,r_R)$;
\item[(ii)] for every set $R$ in $\mathcal R$ its dilated set is defined as 
$R^*=\{x\in S:~d(x,R)<r_R\}$; its right measure satisfies the following inequality:
$$\rho(R^*)\leq \kappa_0\,\rho(R)\,;$$
\item[(iii)] for every set $R$ in $\mathcal R$ there exist 
mutually disjoint sets $R_1,\ldots,R_k$ in $\mathcal R$, with $2\leq k\leq 2^d$, 
such that $R=\bigcup_{i=1}^kR_i$ and $\rho(R_i)=\rho(R)/k$, for $i=1,\ldots,k$. 
\end{itemize}
For any integrable function $f$ and for any $\alpha>0$, $f$ admits a 
\CZ decomposition at level $\alpha$, i.e., a decomposition $f=g+\sum_ib_i$, where $g$ is 
bounded almost everywhere by $\kappa_0 \,\alpha$ and the functions $b_i$ have 
vanishing integral and are supported in \CZ sets $R_i$. The average of 
$|f|$ on each set $R_i$ is comparable with $\alpha$ 
(see \cite[Definition 1.1]{HS} for the details).

Suppose that $p$ is in $(1,\infty]$. By replacing balls with \CZ sets in the classical definition of atoms, 
we say that a function $a$ is a {\emph{$(1,p)$-atom}} if it satisfies the following properties:
\begin{itemize}
\item [(i)] $a$ is supported in a \CZ set $R$;
\item [(ii)]$\|a\|_p\leq \rho (R)^{1/p-1}\,;$ 
\item [(iii)]$\int a\dir =0$\,.
\end{itemize}
Observe that a $(1,p)$-atom is in $L^1$ and it is normalized 
in such a way that its $L^1$-norm does not exceed $1$. 
\begin{defi}
The Hardy space $H^{1,p}$ is the space of all functions $h$ in $ L^1$ 
such that $h=\sum_j \lambda_j\, a_j$, where $a_j$ are $(1,p)$-atoms and $\lambda _j$ 
are complex numbers such that $\sum _j |\lambda _j|<\infty$. We denote by $\|h\|_{H^{1,p}}$ 
the infimum of $\sum_j|\lambda_j|$ over such decompositions. 
\end{defi}
The space $H^{1,p}$ endowed with the norm $\|\cdot\|_{H^{1,p}}$ is a Banach space.

For any $p$ in $ (1,\infty]$ we denote by $H^{1,p}_{\rm{fin}}$ the vector space of all finite linear combinations of $(1,p)$-atoms. Clearly, $H^{1,p}_{\rm{fin}}$ is dense in $H^{1,p}$.

It easily follows from the above definitions that $H^{1,\infty}\subseteq H^{1,p}$,
whenever $p$ is in $ (1, \infty)$. Actually the following theorem holds.
\begin{teo}\label{coincidono}
For any $p$ in $ (1,\infty)$, the spaces $H^{1,p}$ and $H^{1,\infty}$ coincide 
and their norms are equivalent. 
\end{teo}
To prove the Theorem \ref{coincidono} we follow the proof of 
\cite[Theorem A]{CW2}. We shall need the following preliminary result.
\begin{prop}\label{atomo}
Suppose that  $p$ is in $ (1,\infty)$ and $a$ is a $(1,p)$-atom. Then $a$ is in $ H^{1,\infty}$ and there exists a constant $C_{p}$, which depends only on $p$, such that
$$\|a\|_{H^{1,\infty}}\leq C_{p}\,.$$
\end{prop}
\begin{proof}
Let $a$ be a $(1,p)$-atom supported in the \CZ set $R$. We define $b:=\rho(R)\,a$. Note that $b$ is in $ L^p$ and $\|b\|_p\leq \rho(R)^{1/p}$. 

Let $\alpha$ be a positive number such that $\alpha> \max\{ 1,2^{-d/p}\,2^{\frac{1}{p-1}} \} $. We  shall prove that for all $n\in \NN$ there exist functions $a_{j_{\ell}}$, $h_{j_n}$ and \CZ sets $R_{j_{\ell}}$, with $j_{\ell}\in \NN ^{\ell}$, $\ell=0,...,n$, such that
\begin{equation}\label{b}
b=\sum_{\ell=0}^{n-1}2^{\frac{d(\ell +1)}{p}}\,2^{\ell}\,\alpha^{\ell +1}\sum _{j_{\ell}}\rho(R_{j_{\ell}})\,\,a_{j_{\ell}}+\sum _{j_n}h_{j_n}\,,
\end{equation}
where the following properties are satisfied:
\begin{itemize}
\item[(i)] $a_{\jl}$ is a $(1,\infty)$-atom supported in the \CZ set $R_{\jl}$;
\item[(ii)] $h_{j_n}$ is supported in $R_{j_n}$ and $\int h_{j_n}\dir=0$;
\item[(iii)] $\Big(\frac{1}{\rho(R_{j_n})}\int_{R_{j_n}}|h_{j_n}|^p\dir\Big)^{1/p}\leq 2^{dn/p}\,2^n\,\alpha^n$;
\item[(iv)]$\sum_{j_n}\|h_{j_n}\|^p_p\leq 2^{pn}\,\|b\|^p_p$;
\item[(v)] $|h_{j_n}(x)|\leq |b(x)|+2^{dn/p}\,2^n\,\alpha ^n \qquad \forall x\in R_{j_n}$;
\item[(vi)] $\sum_{j_n}\rho(R_{j_n})\leq 2^{d(-n+1)}\,\alpha ^{-np}\,\|b\|^p_p$.
\end{itemize}
We first suppose that the decomposition (\ref{b}) exists and we show that $a$ lies in $H^{1,\infty}$. Set $H_n=\sum_{j_n}h_{j_n}$. By H\"older's inequality
\begin{align*}
\|H_n\|_1&\leq\sum_{j_n}\|h_{j_n}\|_1\leq\sum_{j_n}\rho(R_{j_n})^{1/{p'}}\,\|h_{j_n}\|_p\,,
\end{align*}
where $p'$ is the conjugate exponent of $p$. Now by (iii) and (vi) we have
\begin{align*}
\|H_n\|_1&\leq\sum_{j_n}\rho(R_{j_n})^{1/{p'}}\rho(R_{j_n})^{1/p}\,  2^{dn/p}\,2^n\,\alpha ^n\\
&\leq 2^{d(-n+1)}\,\alpha ^{-np}\,\|b\|^p_p\,2^{dn/p}\,2^n\,\alpha^n\\
&\leq 2^d \big(2\,2^{\frac{d(1-p)}{p}}\,\alpha^{1-p}\big)^n \rho(R)\,.
\end{align*}
Then, since $\alpha >2^{-d/p}\,2^{\frac{1}{p-1}}$, the functions $H_n$ converge to $0$ in $L^1$ when $n$ goes to $\infty$.

This shows that the series $\sum_{\ell=0}^{\infty}2^{\frac{d(\ell +1)}{p}}\,2^{\ell}\,\alpha^{\ell +1}\sum _{j_{\ell}}\rho(R_{j_{\ell}})\,a_{j_{\ell}}$ converges to $b$ in $L^1$. Moreover, by (vi) we deduce that
\begin{align*}
\sum_{\ell=0}^{\infty} 2^{\frac{d(\ell +1)}{p}}\,2^{\ell}\,\alpha^{\ell +1}\,\sum_{j_{\ell}}\rho(R_{j_{\ell}})&\leq \sum_{\ell=0}^{\infty} 2^{\frac{d(\ell +1)}{p}}\,2^{\ell}\,\alpha^{\ell +1}\,2^{d(-\ell +1)}\,\alpha ^{-\ell p}\,\|b\|^p_p\\
&\leq 2^{d(1+1/p)}\,\alpha\, \sum_{\ell=0}^{\infty}\big(2\,2^{\frac{d(1-p)}{p}}\,\alpha^{1-p}\big)^{\ell}\,\rho(R)\\
&=C_{p}\,\rho(R)\,,
\end{align*}
because $\alpha >   2^{-d/p}\,2^{\frac{1}{p-1}}$, where $C_{p}$ depends only on $d,\,p,\, \alpha$.

It follows that $b$ is in $H^{1,\,\infty}$ and $\|b\|_{H^{1,\infty}}\leq C_{p}\,\rho(R)\,.$ Thus $a=\rho(R)^{-1}\,b$ is in  $H^{1,\,\infty}$ and $\|a\|_{H^{1,\infty}}\leq C_{p}$, as required. 

It remains to prove that the decomposition (\ref{b}) exists. This can be done by induction on $n$, following closely the proof of \cite[Theorem A]{CW2}. For the reader's convenience we give the proof in the case $n=1$, and we shall omit the details of the inductive step. 

We construct a partition $\mathcal P$ of $S$ in \CZ sets which contains the set $R$ (see \cite[Proof of 5.1]{HS}). 

{\bf{Step~$n=1$.}} We choose $R_0=R$. Since $\|b\|_p\leq \rho(R)^{1/p}$,
\begin{align*}
\frac{1}{\rho(R)}\int_R |b|^p \dir &\leq \frac{1}{\rho(R)} \,\|b\|_p^p \dir\leq 1\leq \alpha^p\,.
\end{align*}
We split up the set $R$ in at most $2^d$ \CZ subsets. If the average of $|b|^p$ on a subset is greater than $\alpha^p$, then we stop; otherwise we divide again the subset until we find sets on which the average of $|b|^p$ is greater than $\alpha ^p$. We denote by $\mathcal{C}$ the collection of the stopping sets. We distinguish two cases.

{\it{Case}} ~$\mathcal C=\emptyset$. In this case it suffices to define $R_0=R$, $a_0=2^{-d/p}\,\alpha^{-1}\,\rho(R_0)^{-1}\,b$ and $h_i=0$ for all $i\in\NN$. 

{\it{Case~$\mathcal C \neq \emptyset$.}} Let $\mathcal C=\{R_i:~i\in\NN\}$. The average of $|b|^p$ on each set $R_i$ is comparable with $\alpha^p$. Indeed, by construction we have
$$\frac{1}{\rho(R_i)}\int_{R_i}|b|^p\dir>\alpha^p\,.$$
On the other hand, there exists a set $R_i'$, which contains $R_i$, such that $\rho(R_i)\geq\frac{\rho(R_i')}{2^d}$ and $\frac{1}{\rho(R_i')}\int_{R_i'}|b|^p\dir\leq \alpha^p$. It follows that
$$\frac{1}{\rho(R_i)}\int_{R_i}|b|^p\dir\leq \frac{2^d}{\rho(R_i')}\int_{R_i'}|b|^p\dir \leq 2^d \,\alpha^p\,.$$
We define 
\begin{align*}
g(x)&=\begin{cases}
b(x) & {\rm{if}}~ x\notin\bigcup _i{ R_i}\\
\frac{1}{\rho(R_i)}\int_{R_i}b\dir & {\rm{if}}~ x\in R_i
\end{cases}\\
h_i(x)&=\begin{cases}
0 & {\rm{if}}~ x\notin R_i\\
b(x)-\frac{1}{\rho(R_i)}\int_{R_i}b\dir & {\rm{if}}~x\in R_i\qquad\forall i\in \NN\,.
\end{cases}
\end{align*}
Obviously 
\begin{align*}
b=g+\sum_i h_i&=2^{d/p}\,\alpha\,\rho(R_0)\,a_0+\sum_i h_i\,,
\end{align*}
where $a_0=2^{-d/p}\,\alpha^{-1}\,\rho(R_0)^{-1}\,g$. 

The function $a_0$ is supported in $R$ and has vanishing integral. By H\"older's inequality for any $x$ in $R_i$
$$|g(x)|\leq \frac{1}{\rho(R_i)}\int_{R_i}|b|\dir\leq \frac{1}{\rho(R_i)}\rho(R_i)^{1/{p'}}\Big(\int_{R_i}|b|^p\dir\Big)^{1/p}\leq 2^{d/p}\,\alpha\,.$$
If $x$ is in the complement of $\bigcup_i{R_i}$, then all the averages of $|b|^p$ on the sets of the partition $\mathcal P$ which contain $x$ are $\leq \alpha^p$. Thus $|g(x)|\leq \alpha$ for almost every $x$ in the complement of $\bigcup_i{R_i}$. Then $\|a_0\|_{\infty}\leq  \rho(R_0)^{-1}$, so that $a_0$ is a $(1,\infty)$-atom. 

We now verify that the functions $h_i$ satisfy properties (ii)-(vi). Each function $h_i$ is supported in $R_i$ and has vanishing integral. Moreover, by H\"older's inequality
\begin{align}\label{hi}
\|h_i\|_p&\leq \|b\|_{L^p(R_i)}+\rho(R_i)^{1/p}\,\frac{1}{\rho(R_i)}\int_{R_i}|b|\dir \leq 2\,\|b\|_{L^p(R_i)}\,.
\end{align}
Since the sets $R_i$ are mutually disjoint, by summing estimates (\ref{hi}) over $i\in\NN$, we obtain 
$$\sum_i\|h_i\|^p_p\leq 2^p\,\sum_i\|b\|^p_{L^p(R_i)}\leq 2^p\,\|b\|^p_p\,,$$
which proves (iv). From (\ref{hi}) we also have
$$\frac{1}{\rho(R_i)}\int_{R_i}|h_i|^p\dir\leq 2^p\,\frac{1}{\rho(R_i)}\int_{R_i}|b|^p\dir\leq M\,2^p\,\alpha^p\,,$$
which proves (iii). The pointwise estimate (v) of $h_i$ is an easy consequence of H\"older's inequality, since for all $x$ in $ R_i$
\begin{align*}
|h_i(x)|&\leq |b(x)|+\frac{1}{\rho(R_i)}\int_{R_i}|b|\dir\\
&\leq|b(x)|+\rho(R_i)^{-1}\,\rho(R_i)^{1/{p'}}\Big(\int_{R_i}|b|^p\dir\Big)^{1/p}\\
&\leq |b(x)|+M^{1/p}\,2\,\alpha\,.
\end{align*}
It remains to prove property (vi):
\begin{align*}
\sum_i\rho(R_i)&\leq \alpha^{-p}\sum_i\int_{R_i}|b|^p\dir\leq \alpha^{-p}\,\|b\|^p_p\,.
\end{align*}
This concludes the proof of the first step in the case when $\mathcal{C}\neq \emptyset$.

{\bf{Inductive step}}. Suppose that
$$ b=\sum_{\ell=0}^{n-1}2^{\frac{d(\ell +1)}{p}}\,2^{\ell}\,\alpha^{\ell +1}\,\sum _{j_{\ell}}\rho(R_{j_{\ell}})a_{j_{\ell}}+\sum _{j_n}h_{j_n}\,,$$
where the functions $a_{\jl},~h_{\jl}$ and the sets $R_{\jl}$ satisfy properties (i)-(vi). We shall prove that a similar decomposition of $b$ holds with $(n+1)$ in place of $n$. To do so, we decompose each function $h_{j_n}$ by following the same construction used in the case when $n=1$ and the proof of \cite[Theorem A]{CW2}. We omit the details.

\smallskip
This concludes the proof of the proposition.
\end{proof}
Theorem \ref{coincidono} is an easy consequence of Proposition \ref{atomo}.

In the sequel, we denote by $H^1$ the space $H^{1,\infty}$ and by $\|\cdot\|_{H^1}$ the norm $\|\cdot\|_{H^{1,\infty}}$.

\section{The space $BMO$}\label{BMO}
In this section, we introduce the space of functions of bounded mean oscillation and we investigate its properties. For every locally integrable function $f$ and every set $R$ we denote by $f_R$ the average of $f$ on $R$, i.e., $f_R=\frac{1}{\rho(R)}\int_Rf\dir$. 

\begin{defi}
The space $\mathcal{B}\mathcal{M}\mathcal{O}$ is the space of all functions in $L^1_{\rm{loc}}$ such that
$$\sup_R\frac{1}{\rho(R)}\int_R|f-f_R|\dir <\infty\,,$$
where the supremum is taken over all \CZ sets in the family $\mathcal R$. The space $BMO$ is the quotient of $\mathcal{B}\mathcal{M}\mathcal{O}$ module constant functions. It is a Banach space endowed with the norm 
$$\|f\|_{*}=\sup\Big\{\frac{1}{\rho(R)}\int_R|f-f_R|\dir :~R\in\mathcal R \Big\}\,.$$
\end{defi}
We now prove that the functions in $BMO$ satisfy the John--Nirenberg inequality.
\begin{teo}\label{JN}
{\rm{(John--Nirenberg inequality)}} There exist two positive constants $\eta$ and $A$ such that for any $f$ in $BMO$
$$\sup_{R\in\mathcal R}\frac{1}{\rho(R)}\int_R\exp\Big(\frac{\eta}{\|f\|_*}|f-f_R|  \Big)\dir \leq A\,.$$
\end{teo} 
\begin{proof}
We first take $f$ in $L^{\infty}$. Let $R_0$ be a fixed \CZ set. 

Note that $\frac{1}{\rho(R_0)}\int_{R_0}|f-f_{R_0}|\dir\leq  2\,\|f\|_*$. We split up $R_0$ in at most $2^d$ \CZ sets. If the average of $|f-f_{R_0}|$ on a subset is $>2\,\|f\|_*$, then we stop. Otherwise we go on by splitting the sets that we obtain, until we find \CZ sets contained in $R_0$ where the average of $|f-f_{R_0}|$ is $> 2\,\|f\|_*$. Let $\{R_i\}_i$ be the collection of the stopping sets. We have that:
\begin{itemize}
\item[(i)] $|(f-f_{R_0})\chi_{R_0}|\leq 2\|f\|_*$ on $(\cup_i R_i)^c$;
\item[(ii)] $\rho(\cup _iR_i)\leq \frac{\|(f-f_{R_0})\chi_{R_0}\|_1}{2\|f\|_*}\leq \frac{\rho(R_0)\,\|f\|_*}{2\|f\|_*}=\frac{\rho(R_0)}{2}$;
\item[(iii)] $\frac{1}{\rho(R_i)}\int_{R_i}|f-f_{R_0}|\chi_{R_0}\dir> 2\,\|f\|_*$;
\item[(iv)] for each set $R_i$ there exists a \CZ set $R_i'$ which contains $R_i$, whose measure is $\leq 2^d\,\rho(R_i)$ and such that $\frac{1}{\rho(R_i')}\int_{R_i'}|f-f_{R_0}|\chi_{R_0}\dir\leq 2\,\|f\|_*$. Thus 
\begin{align*}
|f_{R_i}-f_{R_0}|&\leq |f_{R_i}-f_{R_i'}|+|f_{R_i'}-f_{R_0}|\\
&\leq \frac{1}{\rho(R_i)}\int_{R_i}|f-f_{R_i'}|\dir+\frac{1}{\rho(R_i')}\int_{R_i'}|f-f_{R_0}|\dir\\
&\leq \frac{2^d}{\rho(R_i')}\int_{R_i'}|f-f_{R_i'}|\dir +2\,\|f\|_*\\
&\leq (2^d+2)\,\|f\|_*\,.
\end{align*}
\end{itemize}
For any positive $t$ we define $F(t)=\sup_R\frac{1}{\rho(R)}\int_R\exp\Big(\frac{t}{\|f\|_*}|f-f_R|\Big)\dir$, which is finite, since we are assuming that $f$ is bounded. From (i)-(iv) above we obtain that
\begin{align*}
\frac{1}{\rho(R_0)}\int_{R_0}\exp\Big(\frac{t}{\|f\|_*}|f-f_{R_0}|\Big)\dir&\leq \frac{1}{\rho(R_0)}\int_{R_0-\cup _iR_i}\nep^{2t}\dir+\\
&+\frac{1}{\rho(R_0)}\sum_i\int_{R_i}\exp\Big(\frac{t}{\|f\|_*}\big(|f-f_{R_i}|+|f_{R_i}-f_{R_0}|\big)\Big)\dir\\
&\leq \nep^{2t}+\frac{1}{\rho(R_0)}\sum_i \int_{R_i}\nep^{(2^d+2)t}\,\exp\Big(\frac{t}{\|f\|_*}|f-f_{R_i}|\Big)\dir\\
&\leq \nep^{2t}+\nep^{(2^d+2)t}\,\frac{1}{\rho(R_0)}\,\frac{\rho(R_0)}{2}\,F(t)\,.
\end{align*}
By taking the supremum over all \CZ sets $R_0$ we deduce that
$$F(t)(1-\nep^{(2^d+2)t}/2)\leq \nep^{2t}\,.$$
This implies that there exists a sufficently small positive $\eta$ such that $F(\eta)\leq C$.

This proves the theorem for all bounded functions. Now let $f$ be in $BMO$ and for $k\in \NN$ define $f_k:S\to \CC$ by
$$f_k(x)=\begin{cases}
f(x) & \text{if $|f(x)|\leq k$}\\
k\,\frac{f(x)}{|f(x)|} & \text{if $|f(x)|> k$}\,.
\end{cases}
$$
Then $\|f_k\|_{\infty}\leq k$ and $\|f_k\|_{*}\leq C\,\|f\|_*$. Moreover $|f_k-f|$ tends monotonically to zero when $k$ tends to $\infty$. We have that
\begin{align*}
\frac{1}{\rho(R)}\int_R\exp\Big(\frac{\eta}{\|f\|_*}|f-f_R|  \Big)\dir&\leq \frac{1}{\rho(R)}\int_R\exp\Big(\frac{\eta}{\|f\|_*}|f-f_k|  \Big)\dir+\\
&+\frac{1}{\rho(R)}\int_R\exp\Big(\frac{\eta}{\|f_k\|_*}|f_k-(f_{k})_{R}|  \Big)\dir\\
&+\frac{1}{\rho(R)}\int_R\exp\Big(\frac{\eta}{\|f\|_*}|(f_k)_{R}-f_R|  \Big)\dir\\
&\leq C+\frac{1}{\rho(R)}\int_R\exp\Big(\frac{\eta}{\|f_k\|_*}|f_k-(f_k)_R|  \Big)\dir\\
&\leq A\,,
\end{align*}
if $k$ is sufficently large. Thus the theorem is proved for all functions in $BMO$.
\end{proof}
A standard consequence of the John--Nirenberg inequality is the following.
\begin{coro}\label{BMOq}
The following hold:
\begin{itemize}
\item[(i)] there exist two positive constants $\eta$ and $A$ such that for any $t>0$
$$\rho\big(\{x\in R:~|f(x)-f_R|>t\,\|f\|_*\}\big)\leq A\,\nep^{-\eta \,t}\,\rho(R)\qquad \forall R\in\mathcal R,\,\forall f\in BMO\,;$$
\item[(ii)] for any $q$ in $ (1,\infty)$ there exists a constant $C_q$, which depends only on $q$, such that 
$$\Big( \frac{1}{\rho(R)}\int_R|f-f_R|^q\dir  \Big)^{1/q}\leq C_q\,\|f\|_*\qquad \forall R\in\mathcal R,\,\forall f\in BMO\,.$$
\end{itemize}
\end{coro}
\begin{proof}
Let $f$ be in $BMO$, $R$ be a \CZ set and take $t>0$. 
 
To prove (i) we observe that by Theorem \ref{JN} 
\begin{align*}
\rho\big(\{x\in R:~|f(x)-f_R|>t\,\|f\|_*\}\big)&=\rho\Big(\{x\in R:~\exp\Big(\frac{\eta}{\|f\|_*}|f(x)-f_R|\Big)>\nep^{\eta\,t}\}\Big)\\
&\leq \frac{\int_R\exp\Big(\frac{\eta}{\|f\|_*}|f-f_{R}|\Big)\dir }{\nep^{\eta\,t}}\\
&\leq A\,{\nep^{-\eta\,t}}\,\rho(R)\,,
\end{align*}
where $\eta$ and $A$ are the constants which appear in Theorem \ref{JN}. 

We now prove (ii).  If $q$ is in $ (1,\infty)$, then there exists $C$ such that $x^q\leq C\,\nep^{\eta\,x}$ for $x>0$. It clearly follows that 
\begin{align*}
\int_R\frac{|f-f_R|^q}{\|f\|_*^q}\dir&\leq \int_R\exp\Big(\frac{\eta}{\|f\|_*}|f-f_{R}|\Big)\dir \leq C\,\rho(R)\,.
\end{align*}
Thus 
$$\Big( \frac{1}{\rho(R)}\int_R|f-f_R|^q\dir  \Big)^{1/q}\leq C_q\,\|f\|_*\,,$$
where $C_q$ only depends on $q$.
\end{proof}
For any $q$ in $ [1,\infty)$ and for every function $f$ in $L^{q}_{\rm{loc}}$ define
$$\|f\|_{q,*}=\sup_{R\in\mathcal R}\Big(\frac{1}{\rho(R)}\int_R|f-f_R|^q\dir \Big)^{1/q}\,,$$
and $BMO_q=\{f\in L^{q}_{\rm{loc}}:~\|f\|_{q,*}<\infty\}\,$. Note that $BMO_1=BMO$ and $\|\cdot\|_{1,*}=\|\cdot\|_*$.

By Corollary \ref{BMOq}(ii), if $f$ is in $BMO$, then $f\in BMO_q$ and $\|f\|_{q,*}\leq C_q\,\|f\|_*$, for any $q$ in $ (1,\infty)$.

Conversely, for any $q$ in $ (1,\infty)$, if $f$ is in $BMO_q$, then trivially $f$ is in $BMO$ and $\|f\|_*\leq \|f\|_{q,*}$.

This means that all the spaces $BMO_q$, with $q$ in $ (1,\infty)$, are equivalent to $BMO$. 

\smallskip
We now prove that the dual space of $H^{1,2}$ may be 
identified with $BMO_{2}$. 
\begin{teo}\label{dualitytheorem}
{\rm{(duality theorem)}}
The following hold:
\begin{itemize}
\item[(i)] for any $f$ in $BMO_{2}$ the functional $\ell$ defined on $H^{1,2}_{\rm{fin}}$ by
$$\ell(g)=\int f\,g\dir\qquad \forall g\in H^{1,2}_{\rm{fin}} \,,$$
extends to a bounded functional on $H^{1,2}$. Furthermore, there exists a constant $C$ such that
$$\|\ell\|_{(H^{1,2})^*}\leq C\,\|f\|_{2,*}\,;$$
\item[(ii)] there exists a constant $C$ such that for any bounded linear functional $\ell$ on $H^{1,2}$ there exists a function $f^{\ell}$ in $BMO_{2}$ such that $\|f^{\ell}\|_{2,*}\leq C\,\|\ell\|_{(H^{1,2})^*}$ and $\ell(g)=\int f^{\ell}\,g\dir$ for any $g$ in $ H^{1,2}_{\rm{fin}}$.
\end{itemize}
\end{teo}
\begin{proof}
The proof of (i) follows the proof of the analogue result in the classical setting \cite{CW2,S}. We omit the details.

We now prove (ii). For any $n\in \NN$ let $R_n$ be the \CZ set $Q_n\times [\nep^{-n},\nep^n]$, where $Q_n$ is a dyadic cube in $\RR^d$ centred at $0$ of side $L_n$, such that $\nep^{2n}\leq L_n< \nep^{8n}$. Obviously, $\bigcup_nR_n=S$.

For any $n\in \NN$ let $X_n$ be the space $L^2_{0}(R_n)$ of all functions in $L^2$ which are supported in $R_n$ and have vanishing integral. The space $(X_n,\|\cdot \|_2)$ 
is a Banach space. We denote by $X$ the space $L^2_{c,0}(S)$ of all functions in $L^2$ with compact support and vanishing integral, interpreted as the strict inductive limit of the spaces $X_n$ (see \cite[II, p. 33]{B} for the definition of the strict inductive limit topology). Observe that $H^{1,2}_{\rm{fin}}$ and $X$ agree as vector spaces.

For any $g$ in $X_n$ the function $\rho(R_n)^{-1/2}\,\|g\|_2^{-1}\,g$ is a $(1,2)$-atom, so that $g$ is in $H^{1,2}$ and $\|g\|_{H^{1,2}}\leq \rho(R_n)^{1/2}\,\|g\|_2$. Hence $X\subset H^{1,2}$ and the inclusion is continuous. 

Now take a bounded linear functional $\ell$ on $H^{1,2}$. Since $X\subset H^{1,2}$, $\ell$ lies in the dual of $X$, i.e, the quotient space $L^{2}_{\rm{loc}}/{\CC}$. Then there exists a function $f^{\ell}$ in $ L^{2}_{\rm{loc}}$ such that
$$\ell(g)=\int f^{\ell}\,g\dir\qquad \forall g\in X\,.$$
It remains to show that $f^{\ell}$ is in $BMO_{2}$. Let $R$ be a \CZ set. For any function $g$ in $X$ which is supported in $R$ the function $\|g\|_2^{-1}\,\rho(R)^{-1/2}\,g$ is a $(1,2)$-atom. Thus
$$\Big|\int_R f^{\ell}\,g\dir\Big|=|\ell(g)|\leq \|\ell\|_{(H^{1,2})^*}\,\|g\|_2\,\rho(R)^{1/2}\,.$$
It easily follows that $\Big(\int_R|f^{\ell}-f^{\ell}_R|^{2}\dir\Big)^{1/2} \leq  \|\ell\|_{(H^{1,2})^*}\,\rho(R)^{1/2}$, i.e., $f^{\ell}$ is in $BMO_{2}$ and $\|f^{\ell}\|_{2,*}\leq \|\ell\|_{(H^{1,2})^*}$. 
\end{proof}
Since we already proved that the space $H^1$ is equivalent to $H^{1,2}$, and the space $BMO$ is equivalent to $BMO_{2}$, the Theorem \ref{dualitytheorem} means that $BMO$ may be identified with the dual space of $H^1$.

\section{$H^1$--$L^1$-boundedness of integral operators}\label{singularintegrals}
We now prove that integral operators whose kernels satisfy a suitable integral H\"ormander condition are bounded from $H^1$ to $L^1$ and from $L^{\infty}$ to $BMO$. Note that the integral H\"ormander condition which we require below is weaker than the integral conditions in the hypothesis of \cite[Theorem 1.2]{HS}. 
\begin{teo}\label{TeolimH1L1}
Let $T$ be a linear operator which is bounded on $L^2$ and admits a locally integrable kernel $K$ off the diagonal which satisfies the condition 
\begin{align}\label{stimaH}
\sup_{R\in \mathcal R} \sup_{y,\,z\in R}\int_{(R^*)^c}|K(x,y)-K(x,z)| \,\dir (x) &<\infty\,.
\end{align} 
Then $T$ extends to a bounded operator from $H^{1}$ to $L^1$.

If the kernel $K$ satisfies the condition
\begin{align}\label{stimaHdual}
\sup_{R\in\mathcal R} \sup_{y,\,z\in R}\int_{(R^*)^c}|K(y,x)-K(z,x)| \,\dir(x) &< \infty \,,
\end{align}
then $T$ extends to a bounded operator from $L^{\infty}$ to $BMO$. 
\end{teo}
\begin{proof}
Suppose that (\ref{stimaH}) is satisfied. We first show that there exists a constant $C$ such that for any $(1,2)$-atom $a$
\begin{equation}\label{Ta}
\|Ta\|_1\leq C\,.
\end{equation}
Let $a$ be a $(1,2)$-atom supported in the \CZ set $R$. Recall that $R\subseteq {B(x_R,\kappa_0 \,r_R)}$, for some $x_R$ in $S$ and $r_R>0$, and that $R^*$ denotes the dilated set $\{x\in S:~d(x,R)<r_R\}$. 

We estimate the integral of $Ta$ on $R^*$ by the Cauchy--Schwarz inequality:
\begin{align}\label{suR^*}
\int_{R^*} |Ta|\dir&\leq \|Ta\|_{2}\,\rho(R^*)^{1/2}\nonumber\\
&\leq\kappa_0^{1/2}\, |\!|\!| T |\!|\!|_{2} \,\|a\|_2\,\rho(R)^{1/2}\nonumber\\
&\leq \kappa_0^{1/2}\, |\!|\!| T |\!|\!|_{2}\,.
\end{align}
We consider the integral of $|Ta|$ on the complement of $R^*$: 
\begin{align}\label{fuoriR^*}
\int_{R^{*c}} |Ta|\dir&\leq \int_{(R^{*})^c}\Big|\int_R K(x,y)\,a(y)\dir(y)  \Big|\dir(x)\nonumber\\
&=
\int_{(R^{*})^c}\Big|\int_R [K(x,y)-K(x,x_R)]\,a(y)\dir(y)  \Big|\dir(x)\nonumber\\
&\leq \int_{(R^{*})^c}\int_R |K(x,y)-K(x,x_R)|\,|a(y)|\dir(y)\dir(x)\nonumber\\
&=\int_R|a(y)|\Big( \int_{(R^{*})^c} |K(x,y)-K(x,x_R)|\dir(x) \Big)\dir(y)\nonumber\\
&\leq  \|a\|_1\,\sup_{y\in R}\int_{(R^{*})^c}|K(x,y)-K(x,x_R)|\dir(x)\nonumber\\
&\leq  C\,.
\end{align}
By (\ref{suR^*}) and (\ref{fuoriR^*}), the inequality (\ref{Ta}) follows. 

We shall deduce from (\ref{Ta}) that $T$ is bounded from $H^1$ to $L^1$. Indeed, by \cite[Remark 1.4]{HS} $T$ is bounded from $L^1$ to the Lorentz space $L^{1,\infty}$. Now take a function $f$ in $H^1$ and suppose that $f=\sum_{j=1}^{\infty }\lambda_ja_j$ is an atomic decomposition of $f$ with $\sum_j|\lambda_j|\sim \|f\|_{H^1}$. Define $f_N=\sum_{j=1}^{N}\lambda_ja_j$. Since $f_N$ converges to $f$ in $L^1$, $Tf_N=\sum_{j=1}^N\lambda_jTa_j$ converges to $Tf$ in $L^{1,\infty}$. On the other hand, by (\ref{Ta})
$$\|Tf_N-\sum_{j=1}^{\infty} \lambda_jTa_j\|_1\leq \sum_{j=N+1}^{\infty}|\lambda_j|\,\|Ta_j\|_1 \leq C\,\sum_{j=N+1}^{\infty}|\lambda_j|\,,$$
so that $Tf_N$ converges to $\sum_{j=1}^{\infty}\lambda_jTa_j$ in $L^1$. This implies that $Tf=\sum_{j=1}^{\infty}\lambda_jTa_j\in L^1$ and $\|Tf\|_1\leq C\,\|f\|_{H^1}\,$, i.e., $T$ is bounded from $H^1$ to $L^1$.

Suppose now that (\ref{stimaHdual}) is satisfied. By arguing as before, we may prove that the adjoint operator $T'$ of $T$ is bounded from $H^{1}$ to $L^1$. By duality it follows that $T$ is bounded from $L^{\infty}$ to $BMO$.
\end{proof}
We can apply the previous results to the multipliers of a distinguished Laplacian $\D$ on $S$. Let 
$$X_0=a\partial_a\qquad X_i=a\partial_{x_i}\qquad i=1,\ldots,d$$
be a basis of left-invariant vector fields of the Lie algebra of $S$ and $\D=-\sum_{i=0}^dX_i^2$ be the corresponding left-invariant Laplacian, which is essentially self-adjoint on $L^2$. In \cite{HS} the authors studied a class of multipliers of $\Delta$. More precisely, let $\psi$ be a function in $C^{\infty}_c(\RR^+)$, supported in $[1/4,4]$, such that
$$\sum_{j\in\ZZ}\psi(2^{-j}\lambda)=1\qquad \forall \lambda\in\RR^+\,.$$
Let $m$ be a bounded measurable function on $\RR^+$. We say that $m$ satisfies a {\emph{mixed Mihlin-H\"ormander condition of order $(s_0,s_{\infty})$}} if
$$\sup_{t<1}\|m(t\cdot)\,\psi(\cdot)\|_{H^{s_0}(\RR)}<\infty \qquad{\rm{and}}\qquad \sup_{t\geq 1}\|m(t\cdot)\,\psi(\cdot)\|_{H^{s_{\infty}}(\RR)}<\infty\,,$$
where $H^s(\RR)$ denotes the $L^2$-Sobolev space of order $s$ on $\RR$. By \cite[Theorem 2.4]{HS} if $m$ satisfies a mixed Mihlin-H\"ormander condition of order $(s_0,s_{\infty})$, with $s_0>3/2$ and $s_{\infty}>\max\{3/2,(d+1)/2  \}$, then the operator $m(\D)$ is bounded from $L^1$ to $L^{1,\infty}$ and bounded on $L^p$, for $p$ in $ (1,\infty)$. We now prove a boundedness result for the same multipliers.
\begin{prop}
Suppose that $s_0>3/2$ and $s_{\infty}>\max\{3/2,(d+1)/2  \}$. If $m$ satisfies a mixed Mihlin--H\"ormander condition of order $(s_0,s_{\infty})$, then the operator $m(\D)$ is bounded from $H^1$ to $L^1$ and from $L^{\infty}$ to $BMO$.
\end{prop}
\begin{proof}
The kernel of the operator $m(\D)$ satisfies the conditions (\ref{stimaH}) and (\ref{stimaHdual}) \cite[Theorem 2.4]{HS}. By Theorem \ref{TeolimH1L1} the result follows.
\end{proof}

\section{Real interpolation}\label {interpolation}
In this section, we study the real interpolation of $H^1$, $BMO$ and the $L^p$ spaces. We first recall some notation of the real interpolation of normed spaces, focusing on the $K$-method. For the details see \cite{BL}. 

Given two compatible normed spaces $X_0$ and $X_1$, for any $t>0$ and for any $x\in X_0+X_1$ we define
$$K(t,x;X_0,X_1)=\inf\{ \|x_0\|_{X_0}+t\|x_1\|_{X_1}:~x=x_0+x_1,\,x_i\in X_i \}\,.$$ 
Take $q$ in $ [1, \infty]$ and $\theta$ in $ (0,1)$. The {\emph{real interpolation space}} $\big[X_0,X_1\big]_{\theta,q}$ is defined as the set of the elements $x\in X_0+X_1$ such that
$$\|x\|_{\theta,q}=\begin{cases}
\Big(\int_0^{\infty}\big[t^{-\theta}\,K(t,x;X_0,X_1)\big]^q \frac{\di t}{t} \Big)^{1/q}&{\rm{if~}} q\in [1,\infty)\\
\|t^{-\theta}\,K(t,x;X_0,X_1)\|_{\infty}&{\rm{if~}} q=\infty\,,
\end{cases}
$$
is finite. The space $\big[X_0,X_1\big]_{\theta,q}$ endowed with the norm $\| \cdot \|_{\theta,q}$ is an exact interpolation space of exponent $\theta$. 

We refer the reader to \cite{Jo} for an overview of the real interpolation results which hold in the classical setting. Our aim is to prove the same results in our context. Note that in our case a maximal characterization of the Hardy space is not avalaible, so that we cannot follow the classical proofs but we shall only use the atomic definition of $H^1$ to prove the results.

\smallskip
We shall first estimate the $K$ functional of $L^{p}$-functions with respect to the couple of spaces $(H^1,L^{p_1})$, with $p_1$ in $(1, \infty]$.
\begin{lem}\label{intpinfty}
Suppose that $1<p< p_1\leq \infty$ and $\frac{1}{p}=1-\theta+\frac{\theta}{p_1}$, with $\theta$ in $ (0,1)$. Let $f$ be in $L^p$. The following hold:
\begin{itemize}
\item[(i)] for every $\lambda>0$ there exists a decomposition $f=g^{\lambda}+b^{\lambda}$ in $L^{p_1}+H^1$ such that
\begin{itemize}
\item[(a)] $\|g^{\lambda}\|_{\infty} \leq C\,\lambda$;
\item[(b)] if $p_1<\infty$, then $\|g^{\lambda}\|_{p_1}^{p_1}\leq C\,\lambda^{p_1-p}\,\|f\|_p^p$;
\item[(c)] $\|b^{\lambda}\|_{H^1}\leq C\,\lambda^{1-p}\,\|f\|_p^p$;
\end{itemize}
\item[(ii)] for any $t>0$, $K(t,f;H^1,L^{p_1})\leq C\,t^{\theta}\,\|f\|_p;$
\item[(iii)] $f\in  [H^1,L^{p_1}]_{\theta,\infty}$ and $\|f\|_{\theta,\infty}\leq C\,\|f\|_p.$ 
\end{itemize}
\end{lem}
\begin{proof}
Let $f$ be in $L^p$. We first prove (i). Given a positive $\lambda$, let $\{R_j\}$ be the collection of sets associated with the \CZ decomposition of $|f|^p$ corresponding to the value $\lambda^p$. We write
$$f=g^{\lambda }+b^{\lambda}=g^{\lambda}+\sum_jb^{\lambda}_j=g^{\lambda}+\sum_j\big(f-f_{R_j}\big)\,\chi_{R_j}\,.$$
We then have
$$\|g^{\lambda}\|_{\infty}\leq C\,\lambda\,,\qquad \frac{1}{\rho(R_j)}\int_{R_j}|f|^p\dir \sim  \lambda^p\qquad{\rm{and}}\qquad |f_{R_j}|\leq C\,\lambda     .$$
If $p_1<\infty$, then
\begin{align*}
\|g^{\lambda}\|_{p_1}^{p_1}&\leq \sum_j\int_{R_j}|f_{R_j}|^{p_1}\dir+\int_{(\bigcup R_j)^c}|f|^{p_1}\dir\\
&\leq C\,\lambda^{p_1}\sum_j\rho(R_j)+\int_{(\bigcup R_j)^c}|f|^{p_1-p}\,|f|^p\dir\\
&\leq C\,\lambda^{p_1}\,\frac{\|f\|_p^p}{\lambda^p}+\lambda^{p_1-p}\,\|f\|_p^p\\
&\leq C\,\lambda^{p_1-p}\,\|f\|_p^p\,.
\end{align*} 
We now prove that $b$ is in $H^{1,p}$. For any $j$, $b^{\lambda}_j$ is supported in $R_j$, has vanishing integral and
$$\Big(\int_{R_j}|b^{\lambda}_j|^p \dir \Big)^{1/p}\leq C\,\rho(R_j)^{1/p}\,\lambda=C\,\lambda\,\rho(R_j)\,\rho(R_j)^{-1+1/p}\,.$$
This shows that $b^{\lambda}_j\in H^{1,p}=H^1$ and $\|b^{\lambda}_j\|_{H^1}\leq C\,\lambda\,\rho(R_j)$. Since $b^{\lambda}=\sum_jb^{\lambda}_j$, $b^{\lambda}$ is in $H^{1}$ and 
$$\|b^{\lambda}\|_{H^{1}}\leq C\,\lambda\,\sum_j\rho(R_j)\leq C\,\lambda\,\frac{\|f\|_p^p}{\lambda^p}\,,$$
as required.

We now prove (ii). Fix $t>0$. For any positive $\lambda$, let $f=g^{\lambda}+b^{\lambda}$ be the decomposition of $f$ in $L^{p_1}+H^1$ given by (i). Thus 
\begin{align*}
K(t,f;H^1,L^{p_1})&=\inf \{\|f_0\|_{H^1}+t\,\|f_1\|_{p_1}:~f=f_0+f_1,\,f_0\in H^1,\,f_1\in L^{p_1}  \}\\
&\leq \inf_{\lambda>0} \big( \|b^{\lambda}\|_{H^1}+t\,\|g^{\lambda}\|_{p_1} \big)\\
&\leq C\,\inf_{\lambda>0}\big(\lambda^{1-p}\,\|f\|_p^{p}+t\,\lambda^{1-p/p_1}\|f\|_p^{p/p_1}  \big)\\
&\leq C\,\|f\|_p^{p/{p_1}}\,\inf_{\lambda>0}\big(\lambda^{1-p}\,\|f\|_p^{p(1-1/{p_1})}+t\,\lambda^{1-p/p_1}  \big)\\
&=C\,\|f\|_p^{p/{p_1}}\,\inf_{\lambda>0} G(t,\lambda)\,,
\end{align*}
where $G(t,\lambda)=\lambda^{1-p}\,\|f\|_p^{p(1-1/{p_1})}+t\,\lambda^{1-p/p_1}$. We now compute the infimum of the function $G$ with respect to the variable $\lambda$. Note that
\begin{align*}
\partial_{\lambda}G(t,\lambda)&=(1-p)\lambda^{-p}\,\|f\|_p^{p(1-1/{p_1})}+(1-p/p_1)t\,\lambda^{-p/p_1}\\
&=\lambda^{-p} \big[(1-p)\,\|f\|_p^{p(1-1/{p_1})}+(1-p/p_1)t\,\lambda^{-p/p_1+p}  \big]\,.
\end{align*}
If $p_1<\infty$, then
$$\inf_{\lambda>0} G(t,\lambda)=G\big(t,C_p\|f\|_p\,t^{p_1/{p-pp_1}}\big)=C_p\,\|f\|_p^{1-p/p_1}\,t^{\frac{p_1(p-1)}{p(p_1-1)}}\,.$$
If $p_1=\infty$, then
$$\inf_{\lambda>0} G(t,\lambda)=G\big(t,C_p\|f\|_p\,t^{-1/p}\big)=C_p\,\|f\|_p\,t^{1-1/p}\,.$$
Hence,
$$K(t,f;H^1,L^{p_1})\leq C_p\,\|f\|_p\,t^{\theta}\,,$$
which proves (ii). This implies that $\|t^{-\theta}\,K(t,f;H^1,L^{p_1})\|_{\infty}\leq C_p\,\|f\|_p$, so 
that $f\in [H^1,L^{p_1}]_{\theta,\infty}$ and $\|f\|_{\theta,\infty}\leq C_p\|f\|_p$, as required in (iii).
\end{proof}
\begin{teo}\label{realintH1Lp2}
Suppose that $1<p<p_1\leq \infty$ and $\frac{1}{p}=1-\theta+\frac{\theta}{p_1}$, with $\theta$ in $(0,1)$. Then 
$$\big[H^1,L^{p_1}  \big]_{\theta,p}=L^p\,.$$
\end{teo}
\begin{proof}
Since $H^1\subset L^1$, we have that $\big[H^1,L^{p_1}  \big]_{\theta,p}\subset \big[L^1,L^{p_1}  \big]_{\theta,p}=L^p$ \cite[Theorem 5.2.1]{BL}. It remains to prove the converse inclusion. 

To do so, we choose $r,\,s,\,\theta_0,\,\theta_1$ such that $1<r<p<s<p_1$, $\frac{1}{r}=1-\theta_0+\frac{\theta_0}{p_1}$ and $\frac{1}{s}=1-\theta_1+\frac{\theta_1}{p_1}$. By Lemma \ref{intpinfty} 
$$L^{r}\subset [H^1,L^{p_1}]_{\theta_0,\infty}\qquad {\rm{and}}\qquad L^{s}\subset [H^1,L^{p_1}]_{\theta_1,\infty}\,.$$
Choose $\eta$ in $(0,1)$ such that $\frac{1}{p}= \frac{1-\eta}{r}+\frac{\eta}{s}$. Then by \cite[Theorem 5.2.1]{BL} 
$$L^p=[L^{r},L^{s}]_{\eta,p}\subset \big[[H^1,L^{p_1}]_{\theta_0,\infty},[H^1,L^{p_1}]_{\theta_1,\infty}   \big]_{\eta,p}\,.$$
It is easy to show that $\theta=(1-\eta)\,\theta_0+\eta\,\theta_1$, so that by the reiteration theorem \cite[Theorem 3.5.3]{BL}
$$\big[[H^1,L^{p_1}]_{\theta_0,\infty},[H^1,L^{p_1}]_{\theta_1,\infty}   \big]_{\eta,p}=[H^1,L^{p_1}]_{\theta,p}\,.$$
Thus $L^p\subset[H^1,L^{p_1}]_{\theta,p}$, as required.
\end{proof}
We shall apply the duality theorem \cite[Theorem 3.7.1]{BL} to deduce a corresponding interpolation result involving $BMO$ and the $L^p$ spaces. To do so, we shall need the following technical lemma.
\begin{lem}\label{dense}
For any $p_1$ in $ (1,\infty)$, $H^1\cap L^{p_1}$ is dense in $H^1$ and in $L^{p_1}$.
\end{lem}
\begin{proof}
Since $H^1_{\rm{fin}}$ is contained in $ H^1\cap L^{p_1}$ and $H^1_{\rm{fin}}$ is dense in $H^1$, it is obvious that $ H^1\cap L^{p_1}$ is dense in $H^1$.

It remains to prove that $H^1\cap L^{p_1}$ is dense in $L^{p_1}$.

Let $L^{\infty}_{c,0}$ denote the space of all functions in $L^{\infty}$ with compact support and integral $0$. If $f$ is in $ L^{\infty}_{c,0}$, then $f$ is in $L^{p_1}$ and $f$ is a multiple of a $(1,\infty)$-atom, so that $f\in H^1$. Thus $ L^{\infty}_{c,0}\subset H^1\cap L^{p_1}$. It is easy to see that
\begin{itemize}
\item[(i)] $L^{\infty}_{c,0}$ is dense in $ L^{\infty}_{c}$ with respect to the $L^{p_1}$-norm;
\item[(ii)] $L^{\infty}_c$ is dense in $L^{p_1}$, since $L^{\infty}_c$ contains $C_c$ which is dense in $L^{p_1}$.
\end{itemize}
Thus $L^{\infty}_{c,0}$ is dense in $L^{p_1}$. This implies that $H^1\cap L^{p_1}$ is dense in $L^{p_1}$, as required.
\end{proof}
\begin{coro}\label{realintBMOLq2}
Suppose that $1<q_1<q<\infty$ and $\frac{1}{q}=\frac{1-\theta}{q_1}$, with $\theta$ in $ (0,1)$. Then
$$\big[L^{q_1},BMO  \big]_{\theta,q}=L^q\,.$$
\end{coro}
\begin{proof}
Let $p$ and $p_1$ be the conjugate exponents of $q$ and $q_1$, respectively. Then $1<p<p_1<\infty$ and $\frac{1}{p}=\theta+\frac{1-\theta}{p_1}$. By Theorem \ref{realintH1Lp2}$$\big[H^1,L^{p_1}  \big]_{1-\theta,p}=L^{p}\,.$$
Since by Lemma \ref{dense} $H^1\cap L^{p_1}$ is dense in $H^1$ and in $L^{p_1}$, we can apply the duality theorem \cite[Theorem 3.7.1]{BL} and conclude that
$$L^{q}=L^{p'}=\big[H^1,L^{p_1}  \big]'_{1-\theta,p}=\big[(H^1)',(L^{p_1})'  \big]_{1-\theta,p'}=\big[BMO,L^{q_1}  \big]_{1-\theta,q}\,.$$
By \cite[Theorem 3.4.1]{BL} it follows that
$$\big[L^{q_1},BMO  \big]_{\theta,q}=\big[BMO,L^{q_1}  \big]_{1-\theta,q}=L^q\,,$$
as required.
\end{proof}
Note that Theorem \ref{realintH1Lp2} also concerns the limit case $p_1=\infty$, showing that $[H^1,L^{\infty}]_{\theta,p}=L^p$, where $1/p=1-\theta$. The Corollary \ref{realintBMOLq2} does not give a result for the limit case $q_1=1$, since it is not possible to deduce it by applying \cite[Theorem 3.7.1]{BL}. To find the interpolation space $[L^1,BMO]_{\theta,q}$, where $1/q=1-\theta$, we shall apply the reiteration theorem by T. Wolff. To do so we shall need the following technical lemma.
\begin{lem}\label{L1capBMO}
For any $p$ in $ (1,\infty)$, $L^1\cap BMO$ is contained in $ L^p$.
\end{lem}
\begin{proof}
Let $p'$ denote the conjugate exponent of $p$. For any $f$ in $L^{p'}$, by applying Lemma \ref{intpinfty}(i) with $\lambda=\|f\|_{p'}$, we may decompose $f$ into a sum $f=g+b$ such that $\|g\|_{\infty}\leq C_p\,\|f\|_{p'}$ and $\|b\|_{H^1}\leq C_p\,\|f\|_{p'}$. Thus $f\in  L^{\infty}+H^1$ and
$$\|f\|_{L^{\infty}+H^1}\leq C_p\,\|f\|_{p'}\,.$$
This proves that $L^{p'}\subset  L^{\infty}+H^1 $. By duality we deduce that $L^p\supset \big(L^{\infty}+H^1\big)'$. It is easy to show that $\big(L^{\infty}+H^1\big)'\supset L^1\cap BMO$, which concludes the proof of the lemma.  

\end{proof}
We can now apply the reiteration theorem by T.~Wolff \cite[Theorem 1]{W} to study the real interpolation between $L^1$ and $BMO$.
\begin{prop}\label{realintL1BMO}
Suppose that $1<q<\infty$ and $\frac{1}{q}={1-\psi}$, with $\psi$ in $ (0,1)$. Then
$$\big[L^{1},BMO  \big]_{\psi,q}=L^{q}\,.$$
\end{prop}
\begin{proof}
We choose $r$ in $(1,q)$. By \cite[Theorem 5.2.1]{BL} and Corollary \ref{realintBMOLq2}
$$\big[L^{1},L^q  \big]_{\phi,r}=L^r\qquad{\rm{and}}\qquad \big[L^{r},BMO  \big]_{\theta,q}=L^q\,,$$
where $\frac{1}{r}=1-\phi+\frac{\phi}{q}$ and $\frac{1}{q}=\frac{1-\theta}{r}$. By Lemma \ref{L1capBMO}, $L^1\cap BMO\subset L^r\cap L^q$; then we can apply the reiteration theorem \cite[Theorem 1]{W} to conclude that
$$ \big[L^{1},BMO  \big]_{\eta,q}=L^q\,,  $$
where $\psi=\frac{\theta}{1-\phi+\phi\theta}$. It is easy to verify that $\frac{1}{q}=1-\psi$, as required.
\end{proof}
We easily deduce a real interpolation result for $H^1$ and $BMO$.
\begin{coro}\label{realintH1BMO}
Suppose that $1<q<\infty$ and $\frac{1}{q}={1-\psi}$, with $\psi$ in $(0,1)$. Then
$$\big[H^{1},BMO  \big]_{\psi,q}=L^{q}\,.$$
\end{coro}
\begin{proof}
Since $H^1\subset L^1$, $\big[H^{1},BMO  \big]_{\psi,q}\subset \big[L^{1},BMO  \big]_{\psi,q}=L^q$. On the other hand, since $L^{\infty}\subset BMO$,
$$L^q=\big[H^{1},L^{\infty}  \big]_{\psi,q}\subset \big[H^{1},BMO  \big]_{\psi,q}\,,$$
as required.
\end{proof}
By applying the reiteration theorem we may also deduce some real interpolation results involving Lorentz spaces. For the definition of the Lorentz spaces $L^{p,q}$ we refer the reader to \cite[Chapter V]{SW}.
\begin{coro}
The following hold:
\begin{itemize}
\item[(i)] if $1<p<p_1\leq \infty$, $1\leq q,\,q_1\leq \infty$, $\theta\in(0,1)$ and $\frac{1}{p}=1-\theta+\frac{\theta}{p_1}$, then 
$$\big[H^1,L^{p_1,q_1}  \big]_{\theta,q}=L^{p,q}\,;$$
\item[(ii)] if $1\leq s,s_1\leq \infty$, $1\leq q_1<q<\infty$, $\theta\in (0,1)$ and $\frac{1}{q}=\frac{1-\theta}{q_1}$, then
$$\big[L^{q_1,s_1},BMO  \big]_{\theta,s}=L^{q,s}\,;$$
\item[(iii)] if $1<q<\infty$, $\theta\in (0,1)$ and $\frac{1}{p}={1-\theta}$, then
$$\big[H^1,BMO  \big]_{\theta,q}=L^{p,q}\,.$$
\end{itemize}
\end{coro}



\end{document}